\documentclass[12pt,reqno]{article}

\usepackage[usenames]{color}
\usepackage{amsmath,amssymb}
\usepackage{graphicx}
\usepackage{amscd}
\usepackage{amsthm}

\usepackage[colorlinks=true,
linkcolor=webgreen,
filecolor=webbrown,
citecolor=webgreen]{hyperref}

\definecolor{webgreen}{rgb}{0,.5,0}
\definecolor{webbrown}{rgb}{.6,0,0}
\definecolor{red}{rgb}{1,0,0}

\usepackage{color}
\usepackage{fullpage}
\usepackage{float}
\usepackage{framed}
\usepackage{amsthm}
\usepackage{amsfonts}
\usepackage{latexsym}
\usepackage{epsf}
\usepackage{mathrsfs}

\newcommand{\seqnum}[1]{\href{http://oeis.org/#1}{\underline{#1}}}

\newcommand{\R}{{\mathbb R}}
\newcommand{\Q}{{\mathbb Q}}

\newcommand{\N}{{\mathbb N}}
\newcommand{\Z}{{\mathbb Z}}

\newcommand{\Phibar}{\overline{\Phi}}

\newcommand{\sg}{\mathrm{sg}}

\def\EMdash{\leavevmode\hbox to 10.6mm{\vrule height .63ex depth -.59ex
    width 10mm\hfill}}

\title{Summation of Certain Infinite Lucas-Related Series}
\author{\sc Bakir FARHI \\
Laboratoire de Mathématiques appliquées \\
Faculté des Sciences Exactes \\
Université de Bejaia, 06000 Bejaia, Algeria \\[1mm]
\href{mailto:bakir.farhi@gmail.com}{bakir.farhi@gmail.com}
}
\date{}

\begin{document}

\theoremstyle{plain}
\newtheorem{thm}{Theorem}[section]
\newtheorem{coll}[thm]{Corollary}
\newtheorem{lemma}[thm]{Lemma}
\newtheorem{prop}[thm]{Proposition}

\theoremstyle{definition}
\newtheorem{defi}[thm]{Definition}
\newtheorem{example}[thm]{Example}
\newtheorem{conj}[thm]{Conjecture}

\theoremstyle{remark}
\newtheorem{remark}[thm]{Remark}
\newtheorem{remarks}[thm]{Remarks}

\numberwithin{equation}{section}

\begin{flushleft}
\textit{J. Integer Sequences}, \\
\textbf{22} (2019), Article 19.1.6. 
\end{flushleft}

\begin{center}
\vskip 1cm{\LARGE\bf Summation of Certain Infinite Lucas-Related \\[2mm] Series}
\vskip 1cm
\large
Bakir Farhi\\
Laboratoire de Math\'ematiques appliqu\'ees\\
Facult\'e des Sciences Exactes\\
Universit\'e de Bejaia, 06000 Bejaia, Algeria\\
\href{mailto:bakir.farhi@gmail.com}{\tt bakir.farhi@gmail.com} \\
\end{center}

\vskip .2 in

\begin{abstract}
In this paper, we find the sums in closed form of certain type of Lucas-related convergent series. More precisely, we generalize the results already obtained by the author in his arXiv paper entitled: ``Summation of certain infinite Fibonacci related series''. 
\end{abstract}

\section{Introduction}
Throughout this paper, we let $\N^*$ denote the set $\N \setminus \{0\}$ of positive integers. We let $\Phi$ denote the golden ratio ($\Phi := \frac{1 + \sqrt{5}}{2}$) and $\Phibar$ its conjugate in the quadratic field $\Q(\sqrt{5})$; that is $\Phibar := \frac{1 - \sqrt{5}}{2} = - \frac{1}{\Phi}$. Further, we let $\sg(x)$ denote the sign of a nonzero real number $x$; that is $\sg(x) = 1$ if $x > 0$ and $\sg(x) = -1$ if $x < 0$.

Let $P , Q \in \R^*$ be fixed such that $\Delta := P^2 - 4 Q > 0$ and consider $\mathscr{L}(P , Q)$ the $\R$-vectorial space of linear sequences ${(w_n)}_{n \in \Z}$ satisfying
\begin{equation}\label{eqa1}
w_{n + 2} = P w_{n + 1} - Q w_n ~~~~~~~~~~ (\forall n \in \Z) .
\end{equation}
The Lucas sequences of the first and second kind are the sequences of $\mathscr{L}(P , Q)$ corresponding respectively to the initial values $(w_0 , w_1) = (0 , 1)$ and $(w_0 , w_1) = (2 , P)$. Those sequences are respectively denoted by ${(U_n)}_n$ and ${(V_n)}_n$. Let $\alpha$ and $\beta$ be the roots of the quadratic equation: $X^2 - P X + Q = 0$ such that $|\alpha| > |\beta|$ (note that $|\alpha| \neq |\beta|$ because by hypothesis $P \neq 0$ and $P^2 - 4 Q > 0$). It is well known that for all $n \in \Z$, we have
\begin{equation}\label{eqa2}
U_n = \frac{\alpha^n - \beta^n}{\alpha - \beta} ~~\text{and}~~ V_n = \alpha^n + \beta^n .
\end{equation}
The connections and likenesses between the Lucas sequences of the first and second kind are numerous; among them, we will use the following which are easy to check
\begin{eqnarray}
V_n & = & U_{n + 1} - Q U_{n - 1} , \label{eqa3} \\
U_{2 n} & = & U_n V_n , \label{eqa4} \\
\beta^n U_m - \beta^m U_n & = & - Q^m U_{n - m} , \label{eqa5} \\
U_n U_{m + r} - U_m U_{n + r} & = & Q^m U_r U_{n - m} , \label{eqa6}
\end{eqnarray}
which hold for any $n , m , r \in \Z$. Note that if we take $(P , Q) = (1 , -1)$, we obtain for the Lucas sequence of the first kind the classical Fibonacci sequence ${(F_n)}_n$ (referenced by \seqnum{A000045} in the OEIS \cite{sloa}) and for the Lucas sequence of the second kind the classical Lucas sequence ${(L_n)}_n$ (referenced by \seqnum{A000032} in the OEIS). Further, if $(P , Q) = (2 , -1)$, we obtain the so called Pell sequences:
$$
U_n = \frac{(1 + \sqrt{2})^n - (1 - \sqrt{2})^n}{2 \sqrt{2}} ,~ V_n = (1 + \sqrt{2})^n + (1 - \sqrt{2})^n
$$
(respectively referenced by \seqnum{A000129} and \seqnum{A002203} in the OEIS). Next, if $(P , Q) = (3 , 2)$, then the sequences obtained are simply:
$$
U_n = 2^n - 1 ,~ V_n = 2^n + 1
$$
(respectively referenced by \seqnum{A000225} and \seqnum{A000051} in the OEIS).

The exact evaluation of infinite Lucas-related series is an old and fascinating subject of study with many still open questions. The particular case dealing with Fibonacci numbers (considered as the most important) has been the subject of several researches that the reader can consult in \cite{jean2, brou, ds, far, hogg, hons, luc, melh, popo, rabi}. For the general case, we just refer the reader to the papers \cite{bg} and \cite{hu} that are close enough to the present work.

This paper is devoted to generalize the results already obtained by the author in \cite{far}, which only concerns the classical Fibonacci sequence, to the larger class of the Lucas sequences of the first kind. We investigate two types of infinite Lucas-related series. The first one consists of the series of one of the two forms:

$$
\sum_{n = 1}^{+ \infty} Q^{a_n} \frac{U_{a_{n + k} - a_n}}{U_{a_n} U_{a_{n + k}}} ~~~~\text{or}~~~~ \sum_{n = 1}^{+ \infty} (-1)^n Q^{a_n} \frac{U_{a_{n + k} - a_n}}{U_{a_n} U_{a_{n + k}}} ,
$$
where $k$ is a positive integer and ${(a_n)}_n$ is a sequence of positive integers tending to infinity with $n$. The second type of series which we consider consists of the series of one of the two forms:
$$
\sum_{n = 1}^{+ \infty} (-1)^n \frac{\beta^{a_n}}{U_{a_n}} ~~~~\text{or}~~~~ \sum_{n = 1}^{+ \infty} \frac{\beta^{a_n}}{U_{a_n}} ,
$$
where ${(a_n)}_n$ is an increasing sequence of positive integers. We show in particular that if $P$ and $Q$ are integers, then some of such series can be transformed on series with rational terms.

\section{The first type of series}\label{sec2}
We begin with the following general result.
\begin{thm}\label{thm 2.1}
Let ${(a_n)}_{n \geq 1}$ be a sequence of positive integers, tending to infinity with $n$, and let $k$ be a positive integer. Then we have
\begin{equation}\label{eq 2.1}
\sum_{n = 1}^{+ \infty} Q^{a_n} \frac{U_{a_{n + k} - a_n}}{U_{a_n} U_{a_{n + k}}} = 
\sum_{n = 1}^{k} \frac{\beta^{a_n}}{U_{a_n}} .
\end{equation}
\end{thm}
To prove this theorem, we need the following lemma.
\begin{lemma}\label{lemma 2.2}
Let ${(x_n)}_{n \geq 1}$ be a convergent real sequence and let $x \in \R$ be its limit. Then for all $k \in \N$, we have
$$
\sum_{n = 1}^{+ \infty} \left(x_{n + k} - x_n\right) = k x - \sum_{n = 1}^{k} x_n .
$$
\end{lemma} 
\begin{proof}
Let $k \in \N$ be fixed. For any positive integer $N$, we have
\begin{eqnarray*}
\sum_{n = 1}^{N} \left(x_{n + k} - x_n\right) & = & \sum_{n = 1}^{N} \sum_{i = 1}^{k} \left(x_{n + i} - x_{n + i - 1}\right) \\
& = & \sum_{i = 1}^{k} \sum_{n = 1}^{N} \left(x_{n + i} - x_{n + i - 1}\right) \\
& = & \sum_{i = 1}^{k} \left(x_{N + i} - x_i\right) \\
& = & \sum_{i = 1}^{k} x_{N + i} - \sum_{i = 1}^{k} x_i .
\end{eqnarray*}
The formula of the lemma immediately follows by tending $N$ to infinity.
\end{proof}
\begin{proof}[Proof of Theorem \ref{thm 2.1}]
Following an idea of Bruckman and Good \cite{bg}, let us apply Lemma \ref{lemma 2.2} for $x_n := \frac{\beta^{a_n}}{\alpha^{a_n} - \beta^{a_n}}$ ($\forall n \geq 1$), which converges to $0$ (since $|\alpha| > |\beta|$). For any $n , k \in \N$, by using \eqref{eqa2} and the fact that $\alpha \beta = Q$, we get
\begin{eqnarray*}
x_{n + k} - x_n & = & \frac{\beta^{a_{n + k}}}{\alpha^{a_{n + k}} - \beta^{a_{n + k}}} - \frac{\beta^{a_n}}{\alpha^{a_n} - \beta^{a_n}} \\[2mm]
& = & \frac{\alpha^{a_n} \beta^{a_{n + k}} - \alpha^{a_{n + k}} \beta^{a_n}}{\left(\alpha^{a_n} - \beta^{a_n}\right) \left(\alpha^{a_{n + k}} - \beta^{n + k}\right)} \\[2mm]
& = & \frac{(\alpha \beta)^{a_n} \left(\beta^{a_{n + k} - a_n} - \alpha^{a_{n + k} - a_n}\right)}{\left(\alpha^{a_n} - \beta^{a_n}\right) \left(\alpha^{a_{n + k}} - \beta^{a_{n + k}}\right)} \\[2mm]
& = & \frac{Q^{a_n} (\beta - \alpha) U_{a_{n + k} - a_n}}{(\alpha - \beta) U_{a_n} \cdot (\alpha - \beta) U_{a_{n + k}}} \\[2mm]
& = & - \frac{1}{\alpha - \beta} Q^{a_n} \frac{U_{a_{n + k} - a_n}}{U_{a_n} U_{a_{n + k}}} .
\end{eqnarray*}
So, Lemma \ref{lemma 2.2} gives
$$
\sum_{n = 1}^{+ \infty} \left(- \frac{1}{\alpha - \beta} Q^{a_n} \frac{U_{a_{n + k} - a_n}}{U_{a_n} U_{a_{n + k}}}\right) = - \sum_{n = 1}^{k} x_n = - \sum_{n = 1}^{k} \frac{\beta^{a_n}}{\alpha^{a_n} - \beta^{a_n}} .
$$
Thus
$$
\sum_{n = 1}^{+ \infty} Q^{a_n} \frac{U_{a_{n + k} - a_n}}{U_{a_n} U_{a_{n + k}}} = \sum_{n = 1}^{k} \frac{\beta^{a_n}}{U_{a_n}} ,
$$
as required. The theorem is proved.
\end{proof}

From Theorem \ref{thm 2.1}, we immediately deduce the following corollary which is already pointed out by Bruckman and Good \cite{bg} and also by Hu et al. \cite{hu}:

\begin{coll}\label{coll 2.3}
Let ${(a_n)}_{n \geq 1}$ be a sequence of positive integers, tending to infinity with $n$. Then we have
\begin{equation}\label{eq 2.2}
\sum_{n = 1}^{+ \infty} Q^{a_n} \frac{U_{a_{n + 1} - a_n}}{U_{a_n} U_{a_{n + 1}}} = \frac{\beta^{a_1}}{U_{a_1}} .
\end{equation}
\end{coll}

If ${(a_n)}_{n \geq 1}$ is an arithmetic sequence of positive integers, Theorem \ref{thm 2.1} gives the following result (already pointed out by Hu et al. \cite{hu} for the case $k = 1$):

\begin{coll}\label{coll 2.4}
Let ${(a_n)}_{n \geq 1}$ be an increasing arithmetic sequence of positive integers and let $r$ be its common difference. Then for any positive integer $k$, we have
\begin{equation}\label{eq 2.3}
\sum_{n = 1}^{+ \infty} \frac{Q^{r (n - 1)}}{U_{a_n} U_{a_{n + k}}} = \frac{Q^{- a_1}}{U_{k r}} \sum_{n = 1}^{k} \frac{\beta^{a_n}}{U_{a_n}} .
\end{equation}
In particular, we have
\begin{equation}\label{eq 2.4}
\sum_{n = 1}^{+ \infty} \frac{Q^{r (n - 1)}}{U_{a_n} U_{a_{n + 1}}} = \frac{(\beta / Q)^{a_1}}{U_{a_1} U_r} .
\end{equation}
\end{coll}

\begin{proof}
To obtain Formula \eqref{eq 2.3}, it suffices to apply Formula \eqref{eq 2.1} of Theorem \ref{thm 2.1} and use that $a_n = r (n - 1) + a_1$ ($\forall n \geq 1$). Then to obtain Formula \eqref{eq 2.4}, we simply set $k = 1$ in formula \eqref{eq 2.3}. This completes the proof. 
\end{proof}

Before continuing with general results, let us give some applications of the preceding results for the usual Fibonacci sequence; so, we must fix $(P , Q) = (1 , -1)$.
\begin{itemize}
\item An immediate application of Formula \eqref{eq 2.3} of Corollary \ref{coll 2.4} gives the following well-known formulas (see, for example, the survey paper of Duverney and Shiokawa \cite{ds}):
\begin{multline}\label{eqza}
\sum_{n = 1}^{+ \infty} \frac{1}{F_{2 n - 1} F_{2 n + 1}} = \frac{\sqrt{5} - 1}{2} ~~~~,~~~~ \sum_{n = 1}^{+ \infty} \frac{1}{F_{2 n} F_{2 n + 2}} = \frac{3 - \sqrt{5}}{2} ~~~~, \\
\sum_{n = 1}^{+ \infty} \frac{(-1)^{n - 1}}{F_n F_{n + 1}} = \frac{\sqrt{5} - 1}{2} ~~~~,~~~~ \sum_{n = 1}^{+ \infty} \frac{(-1)^{n - 1}}{F_n F_{n + 2}} = \sqrt{5} - 2 .~~~~~~~~
\end{multline}
\item Let $k$ be a positive integer and $a \geq 2$ be an integer. By taking in Formula \eqref{eq 2.2} of Corollary \ref{coll 2.3}: $a_n = k a^n$ ($\forall n \geq 1$), we obtain the following formula:
\begin{equation}\label{eqzb}
\sum_{n = 1}^{+ \infty} \frac{F_{(a - 1) k a^n}}{F_{k a^n} F_{k a^{n + 1}}} = \frac{1}{F_{k a} \Phi^{k a}} .
\end{equation}
By taking in addition $a = 2$ in \eqref{eqzb}, we derive the following formula:
\begin{equation}\label{eqzc}
\sum_{n = 0}^{+ \infty} \frac{1}{F_{k 2^n}} = \frac{1}{F_k} + \frac{1}{F_{2 k}} + \frac{1}{F_{2 k} \Phi^{2 k}} ,
\end{equation} 
which is already pointed out by Hoggatt and Bicknell \cite{hogg}. By taking again $k = 1$ in \eqref{eqzc}, we derive the following remarkable formula, discovered since the 1870's by Lucas \cite{luc}:
\begin{equation}\label{eqzd}
\sum_{n = 0}^{+ \infty} \frac{1}{F_{2^n}} = \frac{7 - \sqrt{5}}{2} .
\end{equation}
From Formula \eqref{eqzc}, we deduce that $\sum_{n = 0}^{+ \infty} \frac{1}{F_{k 2^n}} \in \Q(\sqrt{5})$ (for any positive integer $k$). But except the geometric sequences with common ratio $2$, we don't know any other ``regular'' sequence ${(a_n)}_{n \in \N}$ of positive integers, satisfying the property that $\sum_{n = 0}^{+ \infty} \frac{1}{F_{a_n}} \in \Q(\sqrt{5})$. More precisely, we propose the following open question:
\begin{quote}
\noindent \textbf{Open question.} \emph{Is there any linear recurrence sequence ${(a_n)}_{n \in \N}$ of positive integers, which is not a geometric sequence with common ratio $2$ and which satisfies the property that
$$
\sum_{n = 0}^{+ \infty} \frac{1}{F_{a_n}} \in \Q(\sqrt{5}) ?
$$
}
\end{quote}
Next, by taking $a = 3$ in Formula \eqref{eqzb}, we deduce (according to Formula \eqref{eqa4}) the following formula of Bruckman and Good \cite{bg}:
\begin{equation}\label{eqze}
\sum_{n = 1}^{+ \infty} \frac{L_{k 3^n}}{F_{k 3^{n + 1}}} = \frac{1}{F_{3 k} \Phi^{3 k}} .
\end{equation}
By taking $k = 1$ in Formula \eqref{eqze}, we deduce (after some calculations) the formula:
\begin{equation}\label{eqzf}
\sum_{n = 0}^{+ \infty} \frac{L_{3^n}}{F_{3^{n + 1}}} = \frac{\sqrt{5} - 1}{2}
\end{equation}
(also already pointed out by Bruckman and Good \cite{bg}).
\item By taking in Formula \eqref{eq 2.1} of Theorem \ref{thm 2.1}: $a_n = F_n$ ($\forall n \geq 1$) and $k = 1$, we obtain the following:
\begin{equation}\label{eqzg}
\sum_{n = 1}^{+ \infty} (-1)^{F_n} \frac{F_{F_{n - 1}}}{F_{F_n} F_{F_{n + 1}}} = \frac{1 - \sqrt{5}}{2} 
\end{equation}
(also already pointed out by Bruckman and Good \cite{bg}). Next, by taking in Formula \eqref{eq 2.1} of Theorem \ref{thm 2.1}: $a_n = F_n$ ($\forall n \geq 1$) and $k = 2$, we obtain the following:
\begin{equation}\label{eqzh}
\sum_{n = 1}^{+ \infty} (-1)^{F_n} \frac{F_{F_{n + 1}}}{F_{F_n} F_{F_{n + 2}}} = 1 - \sqrt{5} .
\end{equation}
\end{itemize}

Now, with the same context as Theorem \ref{thm 2.1}, the following corollary gives the sums in closed form of the series 
$$
\sum_{n = 1}^{+ \infty} (-1)^n Q^{a_n} \frac{U_{a_{n + k} - a_n}}{U_{a_n} U_{a_{n + k}}} ,
$$
when $k$ is chosen even.

\begin{coll}\label{coll b}
Let ${(a_n)}_{n \geq 1}$ be a sequence of positive integers, tending to infinity with $n$, and let $k$ be a positive integer. Then we have
\begin{equation}\label{eq 2.14}
\sum_{n = 1}^{+ \infty} (-1)^n Q^{a_n} \frac{U_{a_{n + 2 k} - a_n}}{U_{a_n} U_{a_{n + 2 k}}} = - \sum_{n = 1}^{k} Q^{a_{2 n - 1}} \frac{U_{a_{2 n} - a_{2 n - 1}}}{U_{a_{2 n}} U_{a_{2 n - 1}}} .
\end{equation}
\end{coll}

\begin{proof}
By applying Formula \eqref{eq 2.1} of Theorem \ref{thm 2.1} for the sequence ${(a_{2 n})}_{n \geq 1}$ (instead of ${(a_n)}_{n \geq 1}$), we obtain
\begin{equation}\label{eq 2.16}
\sum_{n = 1}^{+ \infty} Q^{a_{2 n}} \frac{U_{a_{2 n + 2 k} - a_{2 n}}}{U_{a_{2 n}} U_{a_{2 n + 2 k}}} = \sum_{n = 1}^{k} \frac{\beta^{a_{2 n}}}{U_{a_{2 n}}} .
\end{equation}
Next, by applying Formula \eqref{eq 2.1} of Theorem \ref{thm 2.1} for the sequence ${(a_{2 n - 1})}_{n \geq 1}$ (instead of ${(a_n)}_{n \geq 1}$), we obtain
\begin{equation}\label{eq 2.17}
\sum_{n = 1}^{+ \infty} Q^{a_{2 n - 1}} \frac{U_{a_{2 n + 2 k - 1} - a_{2 n - 1}}}{U_{a_{2 n - 1}} U_{a_{2 n + 2 k - 1}}} = \sum_{n = 1}^{k} \frac{\beta^{a_{2 n - 1}}}{U_{a_{2 n - 1}}} .
\end{equation}
Now, by subtracting \eqref{eq 2.17} from \eqref{eq 2.16}, we get
\begin{equation*}
\sum_{n = 1}^{+ \infty} \left(Q^{a_{2 n}} \frac{U_{a_{2 n + 2 k} - a_{2 n}}}{U_{a_{2 n}} U_{a_{2 n + 2 k}}} - Q^{a_{2 n - 1}} \frac{U_{a_{2 n + 2 k - 1} - a_{2 n - 1}}}{U_{a_{2 n - 1}} U_{a_{2 n + 2 k - 1}}}\right) = \sum_{n = 1}^{k} \left(\frac{\beta^{a_{2 n}}}{U_{a_{2 n}}} - \frac{\beta^{a_{2 n - 1}}}{U_{a_{2 n - 1}}}\right) , 
\end{equation*}
which we can write as:
$$
\sum_{n = 1}^{+ \infty} (-1)^n Q^{a_n} \frac{U_{a_{n + 2 k} - a_n}}{U_{a_n} U_{a_{n + 2 k}}} = \sum_{n = 1}^{k} \frac{\beta^{a_{2 n}} U_{a_{2 n - 1}} - \beta^{a_{2 n - 1}} U_{a_{2 n}}}{U_{a_{2 n}} U_{a_{2 n - 1}}} .
$$
Finally, since for any $n \in \Z$, we have $\beta^{a_{2 n}} U_{a_{2 n - 1}} - \beta^{a_{2 n - 1}} U_{a_{2 n}} = - Q^{a_{2 n - 1}} U_{a_{2 n} - a_{2 n - 1}}$ (according to Formula \eqref{eqa5}), we conclude that:
$$
\sum_{n = 1}^{+ \infty} (-1)^n Q^{a_n} \frac{U_{a_{n + 2 k} - a_n}}{U_{a_n} U_{a_{n + 2 k}}} = - \sum_{n = 1}^{k} Q^{a_{2 n - 1}} \frac{U_{a_{2 n} - a_{2 n - 1}}}{U_{a_{2 n}} U_{a_{2 n - 1}}} ,
$$
as required. The proof is achieved.
\end{proof}

If ${(a_n)}_{n \geq 1}$ is an arithmetic sequence of positive integers then Corollary \ref{coll b} reduces to the following corollary:


\begin{coll}\label{coll c}
Let ${(a_n)}_{n \geq 1}$ be an increasing arithmetic sequence of positive integers and let $r$ be its common difference. Then for any positive integer $k$, we have
\begin{equation}\label{eq 2.18}
\sum_{n = 1}^{+ \infty} \frac{(-1)^{n - 1} Q^{r (n - 1)}}{U_{a_n} U_{a_{n + 2 k}}} = \frac{U_r}{U_{2 k r}} \sum_{n = 1}^{k} \frac{Q^{2 r (n - 1)}}{U_{a_{2 n}} U_{a_{2 n - 1}}} .
\end{equation}
In particular, we have
\begin{equation}\label{eq e}
\sum_{n = 1}^{+ \infty} \frac{(-1)^{n - 1} Q^{r (n - 1)}}{U_{a_n} U_{a_{n + 2}}} = \frac{1}{U_{a_1} U_{a_2} V_r} .
\end{equation}
\end{coll}

\begin{proof}
To establish Formula \eqref{eq 2.18}, it suffices to apply Corollary \ref{coll b} together with the formula $u_n = r (n - 1) + u_1$ ($\forall n \geq 1$). To establish Formula \eqref{eq e}, we take $k = 1$ in \eqref{eq 2.18} and we use in addition Formula \eqref{eqa4}. 
\end{proof}

\begin{remark}
Let ${(a_n)}_{n \in \N}$ be an increasing arithmetic sequence of natural numbers and let $r$ be its common difference. By Corollary \ref{coll 2.4}, we know a closed form of the sum 
$$
\sum_{n = 1}^{+ \infty} \frac{Q^{r (n - 1)}}{U_{a_n} U_{a_{n + k}}} ~~~~~~~~~~ (k \in \N^*)
$$
and by Corollary \ref{coll c}, we know a closed form of the sum 
\begin{equation}\label{eqaaa}
\sum_{n = 1}^{+ \infty} \frac{(-1)^{n - 1} Q^{r (n - 1)}}{U_{a_n} U_{a_{n + k}}}
\end{equation}
when $k$ is an even positive integer. But if $k$ is an odd positive integer, the closed form of the sum \eqref{eqaaa} is still unknown even in the particular case ``$a_n = n$''. However, we shall prove in what follows that if $a_0 = 0$, there is a relationship between the sums \eqref{eqaaa}, where $k$ lies in the set of the odd positive integers.
\end{remark}

We have the following:

\begin{thm}\label{thm a}
Let
$$
S_{r , k} := \sum_{n = 1}^{+ \infty} \frac{(-1)^{n - 1} Q^{r (n - 1)}}{U_{r n} U_{r (n + k)}} ~~~~~~~~~~ (\forall r , k \in \N^*) .
$$
Then for any positive integer $r$ and any odd positive integer $k$, we have
\begin{equation}\label{eq d}
S_{r , k} = \frac{U_r}{U_{r k}} \left(S_{r , 1} + Q^r \sum_{n = 1}^{(k - 1)/2} \frac{Q^{2 r (n - 1)}}{U_{2 n r} U_{(2 n + 1) r}}\right) .
\end{equation}
\end{thm}

\begin{proof}
Let $r$ and $k$ be positive integers and suppose that $k$ is odd. Because the formula of the theorem is trivial for $k = 1$, we can assume that $k \geq 3$. We have
\begin{eqnarray*}
S_{r , 1} - \frac{U_{r k}}{U_r} S_{r , k} & = & \sum_{n = 1}^{+ \infty} \frac{(-1)^{n - 1} Q^{r (n - 1)}}{U_{r n} U_{r (n + 1)}} - \frac{U_{r k}}{U_r} \sum_{n = 1}^{+ \infty} \frac{(-1)^{n - 1} Q^{r (n - 1)}}{U_{r n} U_{r (n + k)}} \\
& = & \sum_{n = 1}^{+ \infty} (-1)^n Q^{r (n - 1)} \frac{U_{r k} U_{r (n + 1)} - U_r U_{r (n + k)}}{U_r U_{r n} U_{r (n + 1)} U_{r (n + k)}} .
\end{eqnarray*}
But according to Formula \eqref{eqa6} \big(applied to the triplet $(rk , r , rn)$ instead of $(n , m , r)$\big), we have
$$
U_{r k} U_{r (n + 1)} - U_r U_{r (n + k)} = Q^r U_{r n} U_{r (k - 1)} .
$$
Using this, it follows that:
\begin{eqnarray*}
S_{r , 1} - \frac{U_{r k}}{U_r} S_{r , k} & = & \frac{U_{r (k - 1)}}{U_r} \sum_{n = 1}^{+ \infty} \frac{(-1)^n Q^{r n}}{U_{r (n + 1)} U_{r (n + k)}} \\
& = & - Q^r \frac{U_{r (k - 1)}}{U_r} \sum_{n = 1}^{+ \infty} \frac{(-1)^{n - 1} Q^{r (n - 1)}}{U_{b_n} U_{b_{n + k - 1}}} ,
\end{eqnarray*}
where we have put $b_n := r (n + 1)$ ($\forall n \geq 1$). Next, because $(k - 1)$ is even (since $k$ is supposed odd), it follows by Corollary \ref{coll c} that:
\begin{eqnarray*}
S_{r , 1} - \frac{U_{r k}}{U_r} S_{r , k} & = & - Q^r \frac{U_{r (k - 1)}}{U_r} \times \frac{U_r}{U_{(k - 1) r}} \sum_{n = 1}^{(k - 1)/2} \frac{Q^{2 r (n - 1)}}{U_{b_{2 n}} U_{b_{2 n - 1}}} \\
& = & - Q^r \sum_{n = 1}^{(k - 1)/2} \frac{Q^{2 r (n - 1)}}{U_{2 n r} U_{(2 n + 1) r}} .
\end{eqnarray*}
The formula of the theorem follows.
\end{proof}

\begin{remark}
The particular case of Formula \eqref{eq d} corresponding to $(P , Q) = (1 , -1)$ and $r = 1$ is already established by Rabinowitz \cite{rabi}.
\end{remark}
\begin{remark}
In \cite{hu}, Hu et al. established an expression of $S_{r , 1}$ in terms of the values of the Lambert series; and before them, Jeannin \cite{jean2} obtained the same expression in the particular case when $Q = -1$ and $r$ is odd.
\end{remark}

\section{The second type of series}\label{sec3}

We begin this section by dealing with series of the form $\sum_{n \geq 1} (-1)^n \beta^{a_n} / U_{a_n}$, where ${(a_n)}_n$ is an increasing sequence of positive integers. By grouping terms, we transform such series to another type of series whose terms are rational numbers when $P , Q \in \Z$. As we will specify later, some results of this section can be deduced from the results of the previous one by tending the parameter $k$ to infinity. We have the following:

\begin{thm}\label{thm 3.1}
Let ${(a_n)}_{n \geq 1}$ be an increasing sequence of positive integers. Then we have
\begin{equation}\label{eq 3.1}
\sum_{n = 1}^{+ \infty} (-1)^n \frac{\beta^{a_n}}{U_{a_n}} = - \sum_{n = 1}^{+ \infty} Q^{a_{2 n - 1}} \frac{U_{a_{2 n} - a_{2 n - 1}}}{U_{a_{2 n}} U_{a_{2 n - 1}}} .
\end{equation}
\end{thm}

\begin{proof}
The increase of ${(a_n)}_n$ ensures the convergence of the two series in \eqref{eq 3.1}. By grouping terms, we have
\begin{eqnarray*}
\sum_{n = 1}^{+ \infty} (-1)^n \frac{\beta^{a_n}}{U_{a_n}} & = & \sum_{n = 1}^{+ \infty} \left((-1)^{2 n} \frac{\beta^{a_{2 n}}}{U_{a_{2 n}}} + (-1)^{2 n - 1} \frac{\beta^{a_{2 n - 1}}}{U_{a_{2 n - 1}}}\right) \\
& = & \sum_{n = 1}^{+ \infty} \left(\frac{\beta^{a_{2 n}}}{U_{a_{2 n}}} - \frac{\beta^{a_{2 n - 1}}}{U_{a_{2 n - 1}}}\right) \\
& = & \sum_{n = 1}^{+ \infty} \frac{\beta^{a_{2 n}} U_{a_{2 n - 1}} - \beta^{a_{2 n - 1}} U_{a_{2 n}}}{U_{a_{2 n}} U_{a_{2 n - 1}}} \\
& = & \sum_{n = 1}^{+ \infty} \frac{- Q^{a_{2 n - 1}} U_{a_{2 n} - a_{2 n - 1}}}{U_{a_{2 n}} U_{a_{2 n - 1}}} ~~~~~~~~~~ \text{(according to Formula \eqref{eqa5})} ,
\end{eqnarray*}
as required. This achieves the proof of the theorem.
\end{proof}

\begin{remark}
We can also prove Theorem \ref{thm 3.1} by tending $k$ to infinity in Formulas \eqref{eq 2.14} of Corollary \ref{coll b}. To do so, we must previously remark that for any positive integer $n$, we have
$$
\lim_{k \rightarrow + \infty} \frac{U_{a_{n + 2 k} - a_n}}{U_{a_{n + 2 k}}} = \alpha^{- a_n}
$$
(according to \eqref{eqa2}) and that the series of functions 
$$
\sum_{n = 1}^{+ \infty} (-1)^n Q^{a_n} \frac{U_{a_{n + 2 k} - a_n}}{U_{a_n} U_{a_{n + 2 k}}}
$$
(from $\N^*$ to $\R$) converges uniformly on $\N^*$, as we can see for example by observing that:
$$
\left\vert \frac{U_{a_{n + 2 k} - a_n}}{U_{a_{n + 2 k}}}\right\vert = \left\vert \frac{\alpha^{a_{n + 2 k} - a_n}}{\alpha^{a_{n + 2 k}}} \times \frac{1 - \left(\frac{\beta}{\alpha}\right)^{a_{n + 2 k} - a_n}}{1 - \left(\frac{\beta}{\alpha}\right)^{a_{n + 2 k}}}\right\vert \leq \frac{2}{1 - \left\vert \frac{\beta}{\alpha}\right\vert} \alpha^{- a_n} ~~~~~~~~~~ (\forall n , k \in \N^*) .
$$
\end{remark}

According to the closed-form formulas of $U_n$ and $V_n$ (see \eqref{eqa2}), it is immediate that $\lim_{n \rightarrow + \infty} \frac{V_n}{U_n} = \alpha - \beta = \sg(P) \sqrt{\Delta}$. So the approximations $\sqrt{\Delta} \simeq \sg(P) \frac{V_n}{U_n}$ $(n \geq 1)$ are increasingly better when $n$ increases. When supposing $P > 0$ and $Q < 0$, we derive from Theorem \ref{thm 3.1} a curious formula in which the sum of the errors of the all approximations $\sqrt{\Delta} \simeq \sg(P) \frac{V_n}{U_n} = \frac{V_n}{U_n}$ ($n \geq 1$) is transformed to a series whose terms are rational numbers when $P , Q \in \Z$. We have the following:

\begin{coll}\label{coll 3.2}
Suppose that $P > 0$ and $Q < 0$. Then we have
\begin{equation}\label{eq 3.3}
\sum_{n = 1}^{+ \infty} \left|\sqrt{\Delta} - \frac{V_n}{U_n}\right| = 2 \sum_{n = 1}^{+ \infty} \frac{{|\beta|}^n}{U_n} = 2 \sum_{n = 1}^{+ \infty} \frac{{|Q|}^{2 n - 1}}{U_{2 n} U_{2 n - 1}} .
\end{equation}
\end{coll}

\begin{proof}
From the hypothesis $P > 0$ and $Q < 0$, we deduce that $\alpha = \frac{P + \sqrt{\Delta}}{2}$ and $\beta = \frac{P - \sqrt{\Delta}}{2}$ (since $|\alpha| > |\beta|$). Thus $\alpha > 0$, $\beta < 0$ and $\alpha - \beta = \sqrt{\Delta}$. In addition, ($P > 0$ and $Q < 0$) implies that $U_n > 0$ ($\forall n \in \N^*$). Using all these facts together with \eqref{eqa2}, we have for any positive integer $n$:
\begin{multline*}
\left\vert\sqrt{\Delta} - \frac{V_n}{U_n}\right\vert = \left\vert (\alpha - \beta) - \frac{V_n}{U_n}\right\vert = \frac{\left\vert (\alpha - \beta) U_n - V_n\right\vert}{U_n} = \frac{\left\vert(\alpha^n - \beta^n) - (\alpha^n + \beta^n)\right\vert}{U_n} = 2 \frac{{|\beta|}^n}{U_n} ,
\end{multline*}
which gives the first equality of \eqref{eq 3.3}.

The second equality of \eqref{eq 3.3} follows from Theorem \ref{thm 3.1} by taking $a_n = n$ ($\forall n \geq 1$). The proof is complete.
\end{proof}

Before continuing with general results, we will give some important applications of the two preceding results for the usual Fibonacci and Lucas sequences. By taking in Corollary \ref{coll 3.2} $(P , Q) = (1 , -1)$, which corresponds to $(U_n , V_n) = (F_n , L_n)$, we immediately deduce the following:
\begin{coll}\label{coll d}
We have
$$
\sum_{n = 1}^{+ \infty} \left\vert\sqrt{5} - \frac{L_n}{F_n}\right\vert = 2 \sum_{n = 1}^{+ \infty} \frac{1}{F_n \Phi^n} = 2 \sum_{n = 1}^{+ \infty} \frac{1}{F_{2 n} F_{2 n - 1}} .
$$
\end{coll}
Next, the application of Corollary \ref{coll 3.2} for $(P , Q) = (4 , -1)$ can be announced in the following form:
\begin{coll}\label{coll a}
We have
\begin{equation}\label{eq b}
\sum_{r \in \Lambda} \left|\sqrt{5} - r\right| = 2 \sum_{n = 1}^{+ \infty} \frac{1}{F_{3 n} \Phi^{3 n}} = 4 \sum_{n = 1}^{+ \infty} \frac{1}{F_{6 n} F_{6 n - 3}} ,
\end{equation}
where $\Lambda$ denotes the set of the regular continued fraction convergents of the number $\sqrt{5}$.
\end{coll}
\begin{proof}
The continued fraction expansion of the number $\sqrt{5}$ is known to be equal to:
$$
\sqrt{5} = [2 ; 4 , 4 , 4 , \dots]
$$
(see e.g., \cite[page 116]{olds}). From this we derive that for all $n \in \N$, the $n$th-order convergent of $\sqrt{5}$ is $r_n = p_n / q_n$, where ${(p_n)}_{n \in \N}$ and ${(q_n)}_{n \in \N}$ are the sequences of positive integers, satisfying the recurrence relations:
\begin{equation*}
\begin{split}
p_{n + 2} &= 4 p_{n + 1} + p_n \\
q_{n + 2} &= 4 q_{n + 1} + q_n
\end{split} ~~~~~~~~~~ (\forall n \in \N) ,
\end{equation*}
with initial values: $p_0 = 2$, $q_0 = 1$, $p_1 = 9$, $q_1 = 4$. Using those recurrence relations, the sequences ${(p_n)}_n$ and ${(q_n)}_n$ can be extended to negative indices $n$. Doing so, we obtain $p_{-1} = 1$ and $q_{-1} = 0$. It follows from this that $(U_n , V_n) = (q_{n - 1} , 2 p_{n - 1})$ ($n \in \N$) is exactly the couple of Lucas sequences corresponding to $(P , Q) = (4 , -1)$. Since $(P, Q) = (4 , -1)$ gives $\Delta = 20$ and $(\alpha , \beta) = (2 + \sqrt{5} , 2 - \sqrt{5}) = (\Phi^3 , \Phibar^3)$, we deduce by applying Formulas \eqref{eqa2} that we have for all $n \in \N$:
\begin{equation*}
\begin{split}
q_{n - 1} &= \frac{\left(\Phi^3\right)^n - \left(\Phibar^3\right)^n}{\Phi^3 - \Phibar^3} = \frac{\Phi^{3 n} - \Phibar^{3 n}}{2 \sqrt{5}} = \frac{1}{2} F_{3 n} , \\[2mm]
2 p_{n - 1} &= \left(\Phi^3\right)^n + \left(\Phibar^3\right)^n = \Phi^{3 n} + \Phibar^{3 n} = L_{3 n} .
\end{split}
\end{equation*}
Thus
$$
(p_n , q_n) = \left(\frac{1}{2} L_{3 n + 3} , \frac{1}{2} F_{3 n + 3}\right) ~~~~~~~~~~ (\forall n \in \N) .
$$
By applying then Corollary \ref{coll 3.2} for $(P , Q) = (4 , - 1)$, we get
$$
\sum_{n = 1}^{+ \infty} \left\vert\sqrt{20} - \frac{2 p_{n - 1}}{q_{n - 1}}\right\vert = 4 \sum_{n = 1}^{+ \infty} \frac{1}{F_{3 n} \Phi^{3 n}} = 8 \sum_{n = 1}^{+ \infty} \frac{1}{F_{6 n} F_{6 n - 3}} ,
$$
that is
$$
\sum_{n = 0}^{+ \infty} \left|\sqrt{5} - r_n\right| = 2 \sum_{n = 1}^{+ \infty} \frac{1}{F_{3 n} \Phi^{3 n}} = 4 \sum_{n = 1}^{+ \infty} \frac{1}{F_{6 n} F_{6 n - 3}} ,
$$
as required. The corollary is proved.
\end{proof}

Further, by applying Theorem \ref{thm 3.1} for $(P , Q) = (1 , -1)$ and taking successively: $a_n = 2 n$, $a_n = 2 n - 1$ and then $a_n = 2 n + 1$, we respectively obtain the three following formulas:
\begin{equation}\label{eq 3.5}
\begin{split}
\sum_{n = 1}^{+ \infty} \frac{(-1)^{n - 1}}{F_{2 n} \Phi^{2 n}} = \sum_{n = 1}^{+ \infty} \frac{1}{F_{4 n} F_{4 n - 2}} ~,~ 
\sum_{n = 1}^{+ \infty} \frac{(-1)^{n - 1}}{F_{2 n - 1} \Phi^{2 n - 1}} = \sum_{n = 1}^{+ \infty} \frac{1}{F_{4 n - 1} F_{4 n - 3}} ~, \\
\sum_{n = 1}^{+ \infty} \frac{(-1)^{n - 1}}{F_{2 n + 1} \Phi^{2 n + 1}} = \sum_{n = 1}^{+ \infty} \frac{1}{F_{4 n + 1} F_{4 n - 1}} .
\end{split}
\end{equation}
Remark that the addition (side to side) of the two last formulas of \eqref{eq 3.5} gives the first formula of \eqref{eqza}.

Now, by applying another technique of grouping terms, we are going to show that some series of the form $\sum_{n \geq 0} \beta^{a_n} / U_{a_n}$ (where ${(a_n)}_n$ is an arithmetic sequence of a particular type) can be also transformed to series whose terms are rational numbers when $P , Q \in \Z$. We have the following:

\begin{thm}\label{thm 3.3}
For any positive integer $r$, we have
\begin{equation}\label{eq a}
\sum_{n = 0}^{+ \infty} \frac{\beta^{2^r n + 2^{r - 1}}}{U_{2^r n + 2^{r - 1}}} = \sum_{n = 1}^{+ \infty} \frac{Q^{2^{r - 1} n}}{U_{2^r n}} .
\end{equation}
\end{thm}

\begin{proof}
Let $r$ be a positive integer. From the trivial equality of sets:
$$
\left\{2^r n + 2^{r - 1} ~,~ n \in \N\right\} = \left\{2^{r - 1} n ~,~ n \in \N^*\right\} \setminus \left\{2^r n ~,~ n \in \N^*\right\} ,
$$
we get
\begin{eqnarray*}
\sum_{n = 0}^{+ \infty} \frac{\beta^{2^r n + 2^{r - 1}}}{U_{2^r n + 2^{r - 1}}} & = & \sum_{n = 1}^{+ \infty} \frac{\beta^{2^{r - 1} n}}{U_{2^{r - 1} n}} - \sum_{n = 1}^{+ \infty} \frac{\beta^{2^r n}}{U_{2^r n}} \\
& = & \sum_{n = 1}^{+ \infty} \left(\frac{\beta^{2^{r - 1} n}}{U_{2^{r - 1} n}} - \frac{\beta^{2^r n}}{U_{2^r n}}\right) \\
& = & \sum_{n = 1}^{+ \infty} \frac{\beta^{2^{r - 1} n} V_{2^{r - 1} n} - \beta^{2^r n}}{U_{2^r n}}
\end{eqnarray*}
(since $U_{2^r n} = U_{2^{r - 1} n} V_{2^{r - 1} n}$, according to \eqref{eqa4}). But since $\beta^{2^{r - 1} n} V_{2^{r - 1} n} - \beta^{2^r n} =$ \linebreak $\beta^{2^{r - 1} n} \left(\alpha^{2^{r - 1} n} + \beta^{2^{r - 1} n}\right) - \beta^{2^r n} = (\alpha \beta)^{2^{r - 1} n} = Q^{2^{r - 1} n}$ (according to \eqref{eqa2}), we conclude that:
$$
\sum_{n = 0}^{+ \infty} \frac{\beta^{2^r n + 2^{r - 1}}}{U_{2^r n + 2^{r - 1}}} = \sum_{n = 1}^{+ \infty} \frac{Q^{2^{r - 1} n}}{U_{2^r n}} ,
$$
as required. The theorem is proved.
\end{proof}

To finish, let us see what Theorem \ref{thm 3.3} gives for the usual Fibonacci sequence. By taking in Theorem \ref{thm 3.3}: $(P , Q) = (1 , -1)$ and $r = 1$, we obtain the following formula:
\begin{equation}\label{eqa7}
\sum_{n = 0}^{+ \infty} \frac{1}{F_{2 n + 1} \Phi^{2 n + 1}} = \sum_{n = 1}^{+ \infty} \frac{(-1)^{n - 1}}{F_{2 n}} .
\end{equation}
Next, by taking in Theorem \ref{thm 3.3}: $(P , Q) = (1 , -1)$ and $r \geq 2$, we obtain the formula:
\begin{equation}\label{eqa8}
\sum_{n = 0}^{+ \infty} \frac{1}{F_{2^r n + 2^{r - 1}} \Phi^{2^r n + 2^{r - 1}}} = \sum_{n = 1}^{+ \infty} \frac{1}{F_{2^r n}} ~~~~~~~~~~ (\forall r \geq 2) .
\end{equation}
Note that the particular case corresponding to $r = 2$ of the last formula; that is the formula
$$
\sum_{n = 0}^{+ \infty} \frac{1}{F_{4 n + 2} \Phi^{4 n + 2}} = \sum_{n = 1}^{+ \infty} \frac{1}{F_{4 n}} ,
$$
was already pointed out by Melham and Shannon \cite{melh} who proved it by summing both sides of Formula \eqref{eqzc} over $k$, lying in the set of the odd positive integers.

\begin{remark}
Transforming a series of real terms into a series of rational terms could serve, for example, to show the irrationality of the sum of such a series. For example, for ${(w_n)}_n \in \mathscr{L}(P , Q)$, Andr\'e-Jeannin \cite{jean1} proved the irrationality of the series $\sum_{n = 1}^{+ \infty} t^n/w_n$, when $P \in \Z^*$, $Q = \pm 1$, $t \in \Z^*$ and $|t| < |\alpha|$; thus including the series in \eqref{eq a} of Theorem \ref{thm 3.3} (if we assume $P \in \Z^*$ and $Q = \pm 1$). Other similar and related results can also be found in \cite{bor1,bor2,mata,prev}.
\end{remark}

\section{Acknowledgments}
The author would like to thank the editor and the anonymous referee for their helpful comments.

\bigskip
\hrule
\bigskip

\noindent 2010 {\it Mathematics Subject Classification}: 
Primary 11B39; Secondary 97I30.

\noindent \emph{Keywords:} Lucas sequences, Fibonacci numbers, convergent series.

\end{document}